\documentclass[12pt]{article}

\usepackage{amsthm}
\usepackage{amssymb, latexsym, amsmath, amscd, amsfonts}
\usepackage{array, graphicx, lmodern, slantsc}
\usepackage{enumerate}
\usepackage[english]{babel}        
\usepackage[all]{xy}
\CompileMatrices

\usepackage{color}

\def\TC{\protect\operatorname{TC}}
\def\F{\protect\operatorname{Conf}}
\def\cat{\protect\operatorname{cat}}
\def\zcl{\protect\operatorname{zcl}}

\newtheorem{proposition}{Proposition}[section]
\newtheorem{corollary}[proposition]{Corollary}

\newtheorem{theo}[proposition]{Theorem}

\newtheorem{ejem}[proposition]{Example}
\newtheorem{lema}[proposition]{Lemma}

\makeindex

\title{Topological complexity of collision-free multi-tasking motion planning on orientable surfaces}
\author{Jes\'us Gonz\'alez\thanks{Partially supported by Conacyt Research Grant 221221.} \;\,and\; B\'arbara Guti\'errez}
\date{\empty}

\begin{document}

\maketitle

\begin{abstract}
We compute the higher topological complexity of ordered configuration spaces of orientable surfaces, thus extending Cohen-Farber's description of the ordinary topological complexity of those spaces.
\end{abstract}



\tableofcontents

\section{Introduction}

The configuration space of $n$ distinct ordered points of a space $X$, $\F(X, n)$, is the subspace of the $n$-fold cartesian power $X^{\times n}$ given by
$$\F(X, n)=\left\{\rule{0mm}{4mm}(x_1, \ldots, x_n) \in X^{\times n} \;\colon\, x_i \neq x_j \mbox{ whenever } i \neq j\right\}.$$
These spaces play a central role in a number of settings in mathematics, as well as in other areas of science. Our interest lies in topological robotics, where $\F(X,n)$ arises as the model for the state space of a system consisting of $n$ distinct particles moving without collisions on $X$. We focus on the case $X=\Sigma_g$, an orientable surface of genus $g$. Farber's topological complexity ($\TC$) of $\F(\Sigma_g,n)$ has been described in~\cite{CoFa11}. The purpose of this work is to extend Cohen-Farber's results by describing (in Theorem~\ref{main} below) the higher topological complexity of $\F(\Sigma_g,n)$.

\medskip
In preparation for the bulk of the paper, we now recall Farber's notion of topological complexity, together with Rudyak's generalization, the so-called higher topological complexity.

\medskip
The concept of topological complexity (TC) of a space $X$ was introduced early this millennium by Michael Farber as a way to utilize techniques from homotopy theory in order to model and study, from a topological perspective, the motion planning problem in robotics. If $P(X)$ stands for the space of free paths in $X$, then $\TC(X)$ is the reduced Schwarz genus (also known as sectional category) of the fibration $e\colon P(X)\to X\times X$ given by $e(\gamma)=(\gamma(0),\gamma(1))$. We refer the reader to the book~\cite{MR2455573} and the references therein for a discussion of the meaning, relevance, and basic properties of Farber's concept. 

\medskip
The idea was generalized a few years later by Rudyak, who defined in~\cite{Ru10} the $s$-th topological complexity of $X$, $\TC_s(X)$, as the reduced Schwarz genus of the $s$-th fold evaluation map $e_s\colon P(X)\to X^{\times s}$ given by $$e_s(\gamma)=\left(\gamma(0),\gamma\left(\frac{1}{s-1}\right),\gamma\left(\frac{2}{s-1}\right),\ldots,\gamma\left(\frac{s-2}{s-1}\right),\gamma(1)\right).$$
In particular $\TC=\TC_2$. Rudyak's higher topological complexity has been studied systematically in~\cite{bgrt}. The term ``higher'' comes by considering the base space $X^s$ of $e_s$ as the space of sequences of prescribed stages in the motion planning of a robot with state space $X$, while Farber's original definition (with $s = 2$) deals only with the space $X\times X$ of initial-final stages of the robot.

\medskip
We now state our main result.
\begin{theo}\label{main}
The $s$-th topological complexity of $\F(\Sigma_g,n)$ is given by
$$
\TC_s(\F(\Sigma_g,n))=\begin{cases} 
\,s, & \mbox{if $\,g=0$ and $n\leq2;$}\\
\,sn-3, & \mbox{if $\,g=0$ and $n\geq3;$}\\
\,s(n+1)-2, & \mbox{if $\,g=1$ and $n\geq1;$}\\
\,s(n+1), & \mbox{if $\,g\geq2$ and $n\geq1.$}
\end{cases}
$$
\end{theo}
The case $n=1$ in Theorem~\ref{main} has been noted in previous works; see~\cite[Corollary~3.12]{bgrt} for the case $g\leq1$, \cite[Example~16.4]{trmpfarber} for case $g\geq2$ with $s=2$, and~\cite[Proposition~5.1]{gggl} for the case $g\geq2$ with $s\geq3$. This also covers the case $g=0$ with $n=2$ since $\F(S^2,2)$ has the homotopy type of $S^2$. Indeed, by the Gram-Schmidt process, $S^2$ sits inside $F(S^2, 2)$ via the map $x \mapsto (x, -x)$ as a strong deformation retract. Therefore, in what follows we restrict ourselves to the case $n\geq2$, and in fact $n\geq3$ if $g=0$.

\section{Upper bounds}\label{sectupperbounds}
{\bf Genus 0.} For $n\geq3$ the ordered configuration space of $n$ distinct points on the 2-dimensional sphere $S^2$ admits a homotopy decomposition
\begin{equation}\label{eqFS^2}
\F(S^2, n) \simeq \mbox{SO}(3) \times \F(\mathbb{R}^2 - Q_2, n-3)
\end{equation} 
where $Q_2$ is a set of two fixed points on $\mathbb{R}^2$ (see \cite[Theorem 3.1]{CoFa11}, for instance). The higher topological complexity of both factors is known: The topological group $\mbox{SO}(3) \simeq \mathbb{R}\mbox{P}^3$ has
\begin{equation}\label{sharp1}
\TC_s(\mathrm{SO}(3)))= \cat((\mathbb{R}\mbox{P}^3)^{s-1}) = (s-1)\cat(\mathbb{R}\mbox{P}^3)=3(s-1)
\end{equation}
in view of~\cite{LuptonSche13}, whereas~\cite[Theorem 1.3]{JeMa14} gives
\begin{equation}\label{sharp2}
\TC_s(\F(\mathbb{R}^2 - Q_2, n-3))=s(n-3).
\end{equation}
Then~\cite[Theorem 1.3]{JeMa14} gives $\TC_s(\F(S^2,n))\leq sn-3$.

\medskip\noindent{\bf Genus 1.} Since $T=S^1 \times S^1$ is a group, there is a topological decomposition $$\F(T, n) \cong T \times \F(T - Q_1, n-1)$$ where $Q_1$ is a fixed point in $T$. It has been noted that $$\TC_s(T)=2(s-1).$$ On the other hand, $\F(T - Q_1, n-1)$ has the homotopy type of a cell complex of dimension $n-1$ (see~\cite[proof of Theorem~4.1]{CoFa11}). So~\cite[Theorem~3.9]{bgrt} gives $$\TC_s(\F(T - Q_1, n-1))\leq s(n-1),$$
and we get $\TC_s(\F(T,n))\leq s(n+1)-2$.

\medskip\noindent{\bf Genus at least 2.} As noted in the proof of~\cite[Theorem~5.1]{CoFa11}, $\F(\Sigma_g,n)$ has the homotopy type of a cell complex of dimension $n+1$. We thus immediately obtain $\TC_s(\F(\Sigma_g,n))\leq s(n+1)$.

\medskip
The central goal of the paper is to show, by homological methods, that the three upper bounds described in this section are in fact equalities.

\section{Zero divisors via the Totaro spectral sequence}
For a (graded-)commutative unital algebra $A$ over a field $\mathbb{F}$, let $\mu_s\colon A^{\otimes s}\to A$ denote the iterated multiplication map determined by $\mu_s(a_1\otimes \cdots \otimes a_s)=a_1\cdots a_s$. (All tensor products are taken over $\mathbb{F}$.) Elements in the kernel of $\mu_s$ are called $s$-th zero divisors of $A$, and the $s$-th zero-divisors cup-length of $A$, denoted by $\zcl_{s}(A)$, is the maximal number of $s$-th zero-divisors of $A$ having a non-trivial product. Then, for a field $\mathbb{F}$, the $s$-th zero-divisors cup-length of a space $X$ is $\zcl_{s,\mathbb{F}}(X):=\zcl_s(H^*(X;\mathbb{F}))$. In other words, $\zcl_{s,\mathbb{F}}(X)$ is the maximal number of elements in the kernel of the morphism $\Delta_s^*\colon H^*(X^{\times s};R)\to H^*(X;R)$ having a non-trivial product, where $\Delta_s\colon X \to X^{\times s}$ stands for the iterated diagonal.

\begin{proposition}[{\cite[Proposition~3.4]{Ru10}}]\label{cotasseq}
For any field $\mathbb{F}$, $\TC_s(X)\geq \zcl_{s,\mathbb{F}}(X)$.
\end{proposition}

We use Proposition~\ref{cotasseq} to show that each of the upper bounds described in the previous section are sharp. The simplest situation, i.e.~that for $S^2$, is based on the obvious observation that, for algebras $A'$ and $A''$ as above, $A:=A'\otimes A''$ is a (graded-)commutative unital algebra with multiplication $$(a'_1\otimes a''_1)(a'_2\otimes a''_2):=(-1)^{\deg(a''_1)\deg(a'_2)}a'_1a'_2\otimes a''_1a''_2$$ and, in these conditions, $$\zcl_s(A)\geq\zcl_s(A')+\zcl_s(A'').$$ For instance,~(\ref{eqFS^2}) yields
\begin{equation}\label{conclusiong0}
\zcl_{s,\mathbb{F}}(\F(S^2, n))\geq \zcl_{s,\mathbb{F}}(\mathbb{R}\mathrm{P}^3) + \zcl_{s,\mathbb{F}}(\F(\mathbb{R}^2 - Q_2, n-3)).
\end{equation}

\begin{proof}[Proof of Theorem~\ref{main} for $g=0$ and $n\geq3$]
In view of the proof of \cite[Theorem~5.1]{JeMa14}, the assertion in~(\ref{sharp2}) can be strengthened to
$$ \TC_s(\F(\mathbb{R}^2-Q_2,n-3))=s(n-3)=\zcl_{s,\mathbb{Z}_2}(\F(\mathbb{R}^2-Q_2,n-3)),$$ whereas the corresponding equality $$\TC_s(\mathbb{R}\mathrm{P}^3)=3(s-1)=\zcl_{s,\mathbb{Z}_2}(\mathbb{R}\mathrm{P}^3),$$ extending~(\ref{sharp1}), is an easy exercise. Together with~(\ref{conclusiong0}) and Proposition~\ref{cotasseq} we then get
$$
\TC_s(\F(S^2,n))\geq sn-3,
$$
which completes the proof in view of the upper bound given in Section~\ref{sectupperbounds} for $g=0$ and $n\geq3$.
\end{proof}

Proving that the upper bounds in Section~\ref{sectupperbounds} are also optimal for $\Sigma_g$ with $g\geq1$ (and, thus, completing the proof of Theorem~\ref{main}) depends on Lemma~\ref{cotasseq} and a rather explicit calculation to estimate $\zcl_{s,\mathbb{Q}}(\F(\Sigma_g,n))$. We will show
\begin{equation}\label{suficiente-mente}
\zcl_{s,\mathbb{Q}}(\F(\Sigma_g,n))\geq
\begin{cases}
s(n+1)-2,&g=1;\\s(n+1), & g\geq2.
\end{cases}
\end{equation}
As suggested in~(\ref{suficiente-mente}), all cohomology rings  in the remainder of the paper will have rational coefficients.

\medskip
The Leray spectral sequence of the inclusion $\F(M,n)\hookrightarrow M^{\times n}$ is a central tool for computing the rational cohomology ring of the ordered configuration space $\F(M,n)$ when $M$ is an orientable manifold. As shown by Cohen-Taylor~(\cite{MR513543}) and Totaro~(\cite{MR1404924}), the spectral sequence is particularly amenable when $M$ is a complex projective manifold (e.g.~$M=\Sigma_g$). We do not need the whole spectral sequence $\{E(g)_i^{*,*}\}_{i\geq2}$ for $M=\Sigma_g$, only the subalgebra $E(g)_\infty^{*,0}$ of $H^*(\F(\Sigma_g,n))$ detected on the base axis of the spectral sequence, which is described next.

\medskip
Recall that the rational cohomology algebra $H^*(\Sigma_g)$ is the polynomial ring on $2g$ generators $a(p),b(p)\in H^1(\Sigma_g)$ with $1\leq p\leq g$, and an additional generator $\omega\in H^2(\Sigma_g)$ subject to the relations
$$
a(p) a(q) = b(p) b(q) = 0,
\;\quad\mbox{ and }\;\quad
a(p) b(q) =
\begin{cases}\omega, & p=q; \\ 0, & p\neq q,\end{cases}
$$
for any $p,q\in\{1,\ldots, g\}$. Consequently, $H^*(\Sigma_g^{\times n})$ is generated by 1-dimensional classes $a_i(p)$ and $b_i(p)$ ($1\leq i\leq n$ and $1\leq p\leq g$) and by 2-dimensional classes $\omega_i$ ($1\leq i\leq n$), where the subindex $i$ indicates the cartesian factor where the classes come from, subject to the relations
\begin{equation}\label{relsporeje}
a_i(p) a_i(q) = b_i(p) b_i(q) = 0 \quad\mbox{ and }\quad a_i(p) b_i(q) =
\begin{cases}\omega_i, & p=q; \\ 0, & p\neq q,\end{cases}
\end{equation}
for $p,q\in\{1,\ldots,g\}$ and $i\in\{1,\ldots,n\}$. In particular, an additive basis for $H^*(\Sigma_g^{\times n})$ is given by the set $\beta_1$ consisting of the (tensor) products $\mathbf{u}=u_1\cdots u_n$ satisfying
\begin{equation}\label{condbeta1}
u_i\in\{1,a_i(p), b_i(p), \omega_i\colon 1\leq p\leq g\},\mbox{ \ for each $i\in\{1,\ldots,n\}$}.
\end{equation}
Let $D_g$ be the ideal of $H^*(\Sigma_g^{\times n})$ generated by the elements
\begin{equation}\label{relstotaro}
\omega_i+\omega_j+\sum_{p=1}^g\left(b_i(p) a_j(p)-a_i(p) b_j(p)\right)
\end{equation}
for $1\leq i<j\leq n$. In the spectral sequence, $H^*(\Sigma_g^{\times n})$ corresponds to the base $E_2^{*,0}$, and $D_g$ corresponds to the image of the only differentials landing on the base. Therefore:

\begin{lema}[{\cite[Theorem~4 ]{MR1404924}}]\label{totarosubalgebra}
The quotient $E(g)_\infty^{*,0}=H^*(\Sigma_g^{\times n})/D_g$ is a subalgebra of $H^*(\F(\Sigma_g,n))$.
\end{lema}

In particular,~(\ref{suficiente-mente}) will follow once we prove
\begin{equation}\label{suficiente-mente-mente}
\zcl_s(E(g)_\infty^{*,0})\geq
\begin{cases}
s(n+1)-2,&g=1;\\s(n+1), & g\geq2.
\end{cases}
\end{equation}
Actually, a more explicit statement (in terms of a suitably large non-trivial product of $s$-th zero-divisors of $E(g)_\infty^{*,0}$) is given in Theorem~\ref{maintwo} below, which requires some preparatory notation.

\medskip
For $1\leq i\leq n$ and $1\leq p\leq g$, consider the elements $x_i(p),y_i(p)\in E(g)_\infty^{*,0}$ defined by
\begin{itemize}
\item $x_i(p)=a_i(p)$ \ and \ $y_i(p)=b_i(p)$, \ if $p\geq2$, or if $p=1$ with $i=1$;
\item $x_i(1)=a_i(1)-x_1(1)$ \ and \ $y_i(1)=b_i(1)-y_1(1)$, if $i\geq2$.
\end{itemize}
In order to simplify notation, it will be convenient to write $x_i$ and $y_i$ as alternatives for $x_i(1)$ and $y_i(1)$, respectively. Likewise, $a_i$ and $b_i$ will be used as substitutes of $a_i(1)$ and $b_i(1)$, respectively.

\begin{ejem}{\em
The relations~(\ref{relsporeje}) do not hold in $H^*(\Sigma_g^{\times n})$ if the letters $a$ and $b$ are replaced, respectively, by the letters $x$ and $y$. For instance, $a_j(p)a_j(1)=0$, but if $j,p\geq2$, $$x_j(p)x_j(1)=a_j(p)(a_j(1)-a_1(1))=-a_j(p)a_1(1)\neq0.$$ Likewise, $a_j(1)b_j(1)=\omega_j$, while for $2\leq j\leq n$,
\begin{align}
x_j(1)y_j(1)&=(a_j(1)-a_1(1))(b_j(1)-b_1(1))=\omega_j+\omega_1+b_1(1)a_j(1)-a_1(1) b_j(1)\label{deadelantobis}\\
&=\omega_j+\omega_1+y_1(1)(x_j(1)+x_1(1))-x_1(1)(y_j(1)+y_1(1))\nonumber\\&=\omega_j+\omega_1+y_1(1)x_j(1)-\omega_1-x_1(1)y_j(1)-\omega_1\nonumber\\&=\omega_j-\omega_1+y_1(1)x_j(1)-x_1(1)y_j(1).\label{unanueva}
\end{align}
}\end{ejem}

We are now in position to define the $s$-th zero-divisors of $E(g)_\infty^{*,0}$ we need. In fact, we start by describing four types of $s$-zero-divisors of $H^*(\Sigma_g^{\times n})$.
\begin{enumerate}[(I)]
\item For an element $u\in H^*(\Sigma_g^{\times n})$ of positive degree (so $u^2=0$), consider 
the product $\bar{u}\in H^*(\Sigma_g^{\times n})^{\otimes s}$ given by
\begin{eqnarray}
\bar{u}&:=& \displaystyle{\prod_{\ell=2}^{s} (u \otimes 1 \otimes \cdots \otimes 1 \otimes 1 -1 \otimes \cdots \otimes 1 \otimes \overset{\hspace{.5mm}\mbox{\scriptsize$\ell$}}{u} \otimes 1 \otimes \cdots \otimes 1) }\nonumber  \\
&=& \displaystyle{\sum_{\ell=1}^{s} \pm\,\, u \otimes u \otimes \cdots \otimes \overset{\hspace{.2mm}\ell}{1} \otimes u \otimes \cdots \otimes u.}\nonumber
\end{eqnarray}
Here, the index on top of a tensor factor indicates the coordinate where such a factor appears. Note that $\bar{u}$ is a product of $s-1$ $s$-th zero-divisors. We are interested in the product 
\begin{equation}\label{sdbarraswry}
\prod_{i=1}^{n} \bar{x}_i = \sum \pm x_{J_1} \otimes x_{J_2} \otimes \cdots \otimes x_{J_s}
\end{equation}
where the sum is taken over all subsets $J_1, J_2, \ldots, J_s \subseteq\{1,\ldots,n\}$ with the property that every $i \in\{1,\ldots,n\}$ belongs to exactly $s-1$ subsets $J_k$ ($1\leq k\leq s$), and where $$x_{J_t}:=\prod_{i\in J_t}x_i$$ for $t\in\{1,\ldots,s\}$. 
\item For $i \in\{1,\ldots,n\}$, consider the $s$-th zero-divisor
$$
\widetilde{y}_i := y_i \otimes 1 \otimes  \cdots \otimes 1 - 1 \otimes \cdots \otimes 1 \otimes y_i \in H^*(\Sigma_g^{\times n})^{\otimes s}
$$
and the product
\begin{equation}\label{sdcuna}
\prod_{i=1}^{n} \widetilde{y}_i = \!\!\sum_{J \subseteq\{1,\ldots,n\}} \pm\,\, y_{J^c} \otimes 1 \otimes \cdots  \otimes 1 \otimes y_J,
\end{equation}
where $J^c$ stands for the complement of $J$ in $\{1,\ldots,n\}$.
\item For $i\in\{2,\ldots,s-1\}$, consider the $s$-th zero-divisor $$y_{1,i}:= y_1 \otimes 1 \otimes \cdots \otimes 1 -  1 \otimes \cdots \otimes 1 \otimes \overset{i}{y_1} \otimes 1 \otimes \cdots \otimes 1\in H^*(\Sigma_g^{\times n})^{\otimes s}$$ and the product
\begin{equation}\label{sddobles}
\prod_{i=2}^{s-1} y_{1,i} =\!\!\sum_{(\epsilon_1,\ldots,\epsilon_s)\in M_s} \pm y_1^{\epsilon_1} \otimes y_1^{\epsilon_2} \otimes \cdots \otimes y_1^{\epsilon_{s-1}} \otimes 1,
\end{equation}
where $M_s:= \{(\varepsilon_1, \ldots, \varepsilon_{s-1}) \,\colon \exists !\,\, j\in \{1,\ldots,s-1\} \,\,\mbox{with}\,\, \varepsilon_j=0 \,\, \mbox{and} \,\,\varepsilon_i=1 \,\,\mbox{for $i\neq j$}\}$.
\item If $g\geq 2$, consider the $s$-th zero divisors $c,d\in H^*(\Sigma_g^{\times n})^{\otimes s}$ given by
\begin{eqnarray*}
c&=&a_1(2)\otimes 1\otimes1\cdots\otimes 1-1 \otimes a_1(2)\otimes1\cdots\otimes 1,\\
d&=&\begin{cases}
b_1(2)\otimes 1 - 1\otimes b_1(2), &  \mbox{if $s=2$;} \\ b_1(2)\otimes 1 \otimes 1\otimes 1 \otimes \cdots\otimes1 - 1 \otimes 1\otimes b_1(2)\otimes 1 \otimes \cdots\otimes 1,& \mbox{if $s\geq 3$.}
\end{cases}
\end{eqnarray*}
\end{enumerate}

The inequality in~(\ref{suficiente-mente-mente}) and, therefore, Theorem~\ref{main} for $g>0$ are immediate consequences of the following result, whose proof is the central goal in the reminder of the paper.

\begin{theo}\label{maintwo}
\begin{enumerate}[(i)]
\item The image of $\left(\rule{0mm}{3.5mm}\prod_{i=2}^{s-1} y_{1, i} \right) \hspace{-.8mm}\cdot\hspace{-.8mm}\left(\rule{0mm}{3.5mm}\prod_{i=1}^{n} (\bar{x}_i \widetilde{y}_i) \right)$ in $E(1)_\infty^{*,0}$ is non-zero.
\item If $g\geq2$, the image of $c\hspace{-.4mm}\cdot\hspace{-.4mm}d\hspace{-.4mm}\cdot\!\left(\rule{0mm}{3.5mm}\prod_{i=2}^{s-1} y_{1, i} \right)\hspace{-.8mm}\cdot\hspace{-.8mm}\left(\rule{0mm}{3.5mm}\prod_{i=1}^{n} (\bar{x}_i \widetilde{y}_i) \right)$ in $E(g)_\infty^{*,0}$ is non-zero.
\end{enumerate}
\end{theo}

\section{A subquotient of the cohomology of $\F(\Sigma_g,n)$}
The proof of the non-vanishing of the products indicated in Theorem~\ref{maintwo} is greatly simplified by actually working on the quotient of $E(g)_\infty^{*,0}$ obtained by moding out by the ideal generated by the elements
\begin{equation}\label{bunch16}
x_i(p) x_j(q), \quad x_i(p) y_j(q),  \quad y_i(p) y_j(q)
\end{equation}
with $\hspace{.2mm}p,q\in\{2,\ldots,g\}$ and $\hspace{.3mm}i,j\in\{1,\ldots,n\}$, $i\neq j$, and by the elements 
\begin{equation}\label{bunch17}
x_i y_j
\end{equation}
with $i,j\in\{2,\ldots,n\}$, $i\neq j$. Our strategy has two main steps:
\begin{itemize}
\item[S1.] We first get a full additive description of the quotient $A_g$ of $H^*(\Sigma_g^{\times n})$ by the ideal generated by the elements in~(\ref{bunch16}).
\item[S2.] Then we prove that the products indicated in Theorem~\ref{maintwo} are in fact non-trivial in the quotient $B_g$ of $A_g$ by the $A_g$-ideal generated by the elements in~(\ref{relstotaro}) and~(\ref{bunch17}).
\end{itemize}
Furthermore, when dealing with the second step, and in view of the relations coming from~(\ref{bunch16}), the elements in~(\ref{relstotaro}) can safely be replaced by the elements
\begin{equation}\label{relstotarosimplificadas}
\omega_i+\omega_j+b_i(1) a_j(1)-a_i(1) b_j(1)
\end{equation}
for $1\leq i<j\leq n$. It follows that the identity maps on generators induce ring morphisms 
$B_1\to B_2\to B_3\to\cdots$. In particular, item~{\em(i)} in Theorem~\ref{maintwo} becomes a direct consequence of the proof of item~{\em(ii)} in Theorem~\ref{maintwo} sketched in steps S1 and S2 above. Accordingly, we assume $g\geq2$ in the remainder of the section.

\medskip
Step S1 above is accomplished in either of the next two results.
\begin{lema}\label{basebeta2}
An additive basis of $A_g$ is given by the set $\beta_2$ consisting of the $A_g$-images of the monomials $u_1\cdots u_n\in H^*(\Sigma_g^{\times s})$ satisfying the following two conditions:
\begin{enumerate}[(i)]
\item\label{unoi} For each $i\in\{1,\ldots,n\}$, the factor $u_i$ belongs to $\{1,a_i(p), b_i(p), \omega_i\colon 1\leq p\leq g\}$.
\item\label{dosi} At most one of $u_1,\ldots, u_n$ belongs to $\{a_i(p),b_i(p),\omega_i\colon 1\leq i\leq n\mbox{ and }\,2\leq p\leq g\}$.
\end{enumerate}
\end{lema}
\begin{proof}
Recall the additive basis $\beta_1$ of $H^*(\Sigma_g^{\times n})$ consisting of the products 
$\mathbf{u}=u_1\cdots u_n$ satisfying~(\ref{condbeta1})---i.e.~condition~{\em(\ref{unoi})} of the lemma. In these terms, the defining relations for $A_g$ coming from the elements in~(\ref{bunch16}) take the form
\begin{equation}\label{tomalaforma}
a_i(p)a_j(q)=a_i(p)b_j(q)=b_i(p)b_j(q)=0
\end{equation}
for $p,q\in\{2,\ldots,g\}$ and $i,j\in\{1,\ldots,n\}$ with $i\neq j$. The fact that $a_i(p)b_i(p)=\omega_i$ for any pair $(i,p)$ then implies that a basis element $u_1\cdots u_n\in\beta_1$ vanishes in $A_g$ whenever condition~{\em(\ref{dosi})} of the lemma fails. The result follows since, for any $\mathbf{u}=u_1\cdots u_n\in\beta_1$, the $\mathbf{u}$-multiple (in $H^*(\Sigma_g^{\times n})$) of any of the elements
$$
a_i(p)a_j(q), \quad a_i(p)b_j(q), \quad b_i(p)b_j(q)
$$
as in~(\ref{tomalaforma}) either vanishes or, else, reduces (up to a sign) to an element of $\beta_1$ for which~{\it(\ref{dosi})} fails. For instance, in the case of an element of the form $a_i(p_0)b_j(q_0)$ with $p_0,q_0\geq2$ and $i\neq j$, $$u_1\cdots u_n\cdot a_i(p_0)b_j(q_0)=\pm u_1\cdots \widehat{u}_i\cdots \widehat{u}_j\cdots u_n\left(u_i\hspace{.2mm}a_i(p_0)\right)\left(u_j\hspace{.2mm}b_j(q_0)\right)$$ which, in view of~{\it(\ref{unoi})}, is either zero or, else, an element of $\beta_1$ not satisfying~{\it(\ref{dosi})}.
\end{proof}

\begin{corollary}\label{baseprimamodificada}
An additive basis of $A_g$ is given by the set $\beta'_2$ consisting of the $A_g$-images of the monomials $v_1\cdots v_n\in H^*(\Sigma_g^{\times n})$ satisfying the following two conditions:
\begin{enumerate}[(i)]
\item\label{unoiprima} For each $i\in\{1,\ldots,n\}$, the factor $v_i$ belongs to $\{1,x_i(p), y_i(p), \omega_i\colon 1\leq p\leq g\}$.
\item\label{dosiprima} At most one of $v_1,\ldots, v_n$ belongs to $\{x_i(p),y_i(p),\omega_i\colon 1\leq i\leq n\mbox{ and }\,2\leq p\leq g\}$.
\end{enumerate}
\end{corollary}
\begin{proof}
Since $\beta_2$ and $\beta'_2$ have the same cardinality, it is enough to prove that (the obvious preimage in $H^*(\Sigma_g^{\times n})$ of) each element of $\beta_2$ can be expressed, modulo the $H^*(\Sigma_g^{\times n})$-ideal $\mathcal{I}$ generated by elements in~(\ref{bunch16}), in terms of (the obvious preimages of) elements of $\beta'_2$.

\smallskip Fix $u_1\cdots u_n\in\beta_2$ and let $J\subseteq\{1,\ldots,n\}$ be the set of indices $i$ for which $u_i\in\{a_i,b_i\}$. Thus $u_i=1$ for all $i\in\{1,\ldots,n\}-J$ with the possible exception of a single element $i_0\in\{1,\ldots,n\}-J$ for which $u_{i_0}\in\{a_{i_0}(p),b_{i_0}(p),\omega_{i_0}\colon 2\leq p\leq g\}$. In what follows we let
$$
z_i=\begin{cases}
x_i, & \mbox{if $u_i=a_i$;}\\
y_i, & \mbox{if $u_i=b_i$,}
\end{cases}
\qquad\mbox{and}\qquad
u'_i=\begin{cases}
x_1, & \mbox{if $u_i=a_i$;}\\
y_1, & \mbox{if $u_i=b_i$,}
\end{cases}$$
for $i\in J$.

\medskip\noindent 
{\bf Case 1.} Assume $1\in J$, and that the possible exceptional $i_0$ does not hold. Then
\begin{align*}
u_1\cdots u_n=\prod_{i\in J}u_i=u_1\!\!\!\prod_{i\in J-\{1\}}\!\!\left(z_i+u'_i\right)=z_1\left(\,\prod_{i\in J-\{1\}}z_i+\sum_{j\in J-\{1\}}\left(\pm u'_j\!\!\!\prod_{\ i\in J-\{1,j\}}\!\!z_i\right)\right)\!
\end{align*}
which, as an element of $H^*(\Sigma_g^{\times n})$, is a linear combination of elements in $\beta'_2$. (Note that some terms in the latter summation may vanish as the corresponding factor $z_1u'_j$ may be trivial.)

\medskip\noindent 
{\bf Case 2.} Assume $1\not\in J$, and that the possible exceptional $i_0$ does not hold. Then
\begin{align*}
u_1\cdots u_n&\;=\;\prod_{i\in J}u_i\;=\;\prod_{i\in J}\left(z_i+u'_i\right)\\&\;=\prod_{i\in J}z_i+\sum_{j\in J}\left(\pm u'_j\!\!\!\prod_{\ i\in J-\{j\}}\!\!z_i\right)+\!\!\!\sum_{\mbox{\tiny$
\begin{matrix}
j_1,j_2\in J \\ j_1<j_2
\end{matrix}
$}}\!\left(\pm u'_{j_1}u'_{j_2}\!\!\!\prod_{\ i\in J-\{j_1,j_2\}}\!\!z_i\right)\!,
\end{align*}
again a linear combination in $H^*(\Sigma_g^{\times n})$ of elements in $\beta'_2$. (As in the previous case, some elements in the latter summation may vanish as the corresponding factor $u'_{j_1}u'_{j_2}$ may be trivial.)

\medskip\noindent 
{\bf Case 3.} Assume $1\in J$, and that the possible exceptional case $i_0$ holds ($i_0\neq1$ is forced). Then
\begin{align*}
u_1\cdots u_n&=\pm u_1u_{i_0}\!\!\!\prod_{i\in J-\{1\}}\!\!\left(z_i+u'_i\right)=\pm z_1u_{i_0}\left(\,\prod_{i\in J-\{1\}}z_i+\sum_{j\in J-\{1\}}\left(\pm u'_j\!\!\!\prod_{\ i\in J-\{1,j\}}\!\!z_i\right)\right)\\
&=\pm u_{i_0}\left(\,\prod_{i\in J}z_i+\sum_{j\in J-\{1\}}\left(\pm u'_j\!\!\!\prod_{\ i\in J-\{j\}}\!\!z_i\right)\right).
\end{align*}
This time the term $u_{i_0}\prod_{i\in J}z_i$ lies in $\beta'_2$ (for $a_{i_0}(p)=x_{i_0}(p)$ and $b_{i_0}(p)=y_{i_0}(p)$ if $2\leq p\leq g$), while each of the terms
\begin{equation}\label{eachterm}
u_{i_0}u'_j\prod_{i\in J-\{j\}}z_i
\end{equation}
(with $j\in J-\{1\}$) vanishes modulo $\mathcal{I}$. In fact, the factor $u'_jz_1$ in~(\ref{eachterm}) is 0 or $\omega_1$; either way $u_{i_0}u'_jz_1\equiv0\bmod \mathcal{I}$ since, after all, $\omega_t=a_t(2)b_t(2)=x_t(2)y_t(2)$ for $t=1,i_0$ (c.f.~the assertion following~(\ref{tomalaforma}) in the proof of Lemma~\ref{basebeta2}).

\medskip\noindent 
{\bf Case 4.} Assume that the exceptional $i_0$ holds with $i_0=1$ (so $1\not\in J$). Then
\begin{align*}
u_1\cdots u_n=\pm u_{i_0}\prod_{i\in J}u_i=\pm u_{i_0}\prod_{i\in J}(z_i+u'_i)=\pm u_{i_0}\left(\,\prod_{i\in J}z_i+\sum_{j\in J}\pm u'_j\!\!\prod_{i\in J-\{j\}}\!\!z_i\,\right).
\end{align*}
As in Case 3, the term $u_{i_0}\prod_{i\in J}z_i$ lies in $\beta'_2$. But now, each of the terms $u_{i_0}u'_j\prod_{i\in J-\{j\}}z_i$ (with $j\in J$) vanishes directly in $H^*(\Sigma_g^{\times n})$. Indeed, $u_{i_0}u'_j=0$ as $u'_j\in\{a_1,b_1\}$ and $u_{i_0}\in\{a_1(p),b_1(p), \omega_1\colon 2\leq p\leq g\}$.

\medskip\noindent 
{\bf Case 5.} Assume $1\not\in J$, and that the exceptional $i_0$ holds with $i_0\neq1$. Then
\begin{align*}
u_1\cdots u_n=\pm u_{i_0}\prod_{i\in J}(z_i+u'_i)=\pm u_{i_0}\left(\,\prod_{i\in J}z_i+\sum_{j\in J}\pm u'_j\!\!\!\prod_{i\in J-\{j\}}\!\!\!z_i+\sum_{\mbox{\tiny$
\begin{matrix}
j_1,j_2\in J\\j_1<j_2
\end{matrix}
$}}\pm u'_{j_1}u'_{j_2}\!\!\!\prod_{i\in J-\{j_1,j_2\}}\!\!\!z_i\right).
\end{align*}
Now the term $u_{i_0}\prod_{i\in J}z_i$ as well as the terms $u_{i_0}u'_j\prod_{i\in J-\{j\}}z_i$ (with $j\in J$) lie in $\beta'_2$. The rest of the terms in the previous displayed equation vanish modulo $\mathcal{I}$ just as in Case 3 above.
\end{proof}

We now start working toward the completion of step S2. Recall that $B_g$ is the quotient of $A_g$ by the ideal generated by the elements in~(\ref{bunch17}) and~(\ref{relstotarosimplificadas}). As noted in~(\ref{deadelantobis}), the case $1=i<j\leq n$ of the latter generators is given by $x_jy_j$, whereas for the case $2\leq i<j\leq n$ we have
\begin{eqnarray*}
\omega_i+\omega_j+b_ia_j-a_ib_j&=&(a_i-a_j)(b_i-b_j)\\&=&(x_i-x_j)(y_i-y_j)\\&=&x_iy_i+x_jy_j-x_iy_j-x_jy_i.
\end{eqnarray*}
Consequently we will work with the simplified presentation
\begin{equation}\label{ecopre}
B_g=A_g/\mathcal{J}_g
\end{equation}
where $\mathcal{J}_g$ is the $A_g$-ideal generated by the products $x_iy_j$ with $i,j\in\{2,\ldots,n\}$.

\medskip
A key ingredient for step S2 is given by the next result, whose proof is deferred to the next section of the paper.

\begin{theo}\label{indlin}
The $B_g$-images of the two elements $\,\omega_1x_2\cdots x_n,\,\omega_1y_2\cdots y_n\in H^*(\Sigma_g^{\times n})$ are distinct and, in fact, linearly independent.
\end{theo}
\begin{proof}[Proof of item {(\it ii)} of Theorem~\ref{maintwo} for $s=2$]
As advertised at the beginning of this section, it suffices to work in $B_g$. Direct calculation gives $c d x_1 y_1=2\omega_1\otimes\omega_1$ and (by induction on $n\geq2$, keeping in mind the relations in $B_g$ coming from the ideal $\mathcal{J}_g$) $$cd\left(\,\prod_{i=1}^n(\bar{x}_i\widetilde{y}_i)\right)=2\omega_1\otimes\omega_1\left(\,\pm \,x_2\cdots x_n\otimes y_2\cdots y_n\,\pm\, y_2\cdots y_n\otimes x_2\cdots x_n\rule{0mm}{4mm}\right),$$ which is non-zero in $B_g$ in view of Theorem~\ref{indlin}. (Note that the factor~(\ref{sddobles}) degenerates to~1.)
\end{proof}

The proof of item {(\it ii)} of Theorem~\ref{maintwo} for $s\geq3$ is slightly more involved, partly due to the presence of the factor~(\ref{sddobles}), and partly because of the resulting larger combinatorial objects to deal with. Actually, the main reason for the $s$-th zero-divisor $d$ to be slightly different for $s\geq3$ is to simplify the proof argument. 

\begin{proof}[Proof of item {(\it ii)} of Theorem~\ref{maintwo} for $s\geq3$]
Up to a sign, the product under consideration, $cd\left(\,\prod_{i=2}^{s-1} y_{1, i} \right) \left(\,\prod_{i=1}^{n} (\bar{x}_i \widetilde{y}_i) \right)$, is a sum running over the subsets $J,J_1, J_2, \ldots, J_s$ of $\{1,\ldots,n\}$ specified in~(\ref{sdbarraswry}) and~(\ref{sdcuna}), over the tuples $(\epsilon_1, \epsilon_2, \ldots, \epsilon_{s-1})\in M_s$ specified in~(\ref{sddobles}), and over the pairs $(\alpha_1,\alpha_2)$ and $(\beta_1,\beta_3)$ satisfying $\{\alpha_1,\alpha_2\}=\{0,1\}=\{\beta_1,\beta_3\}$. The term $T$ corresponding to such a data takes the form indicated below, depending on the value of~$s$.

\smallskip\noindent
$\bullet$ If $s\geq5$,
$$
\pm \, a_1(2)^{\alpha_1} b_1(2)^{\beta_1} y_1^{\epsilon_1}y_{J^c}x_{J_1}\otimes a_1(2)^{\alpha_2}y_1^{\epsilon_2}x_{J_2} \otimes b_1(2)^{\beta_3}y_1^{\epsilon_3}x_{J_3} \otimes y_1^{\epsilon_4}x_{J_4}\otimes \cdots \otimes y_1^{\epsilon_{s-1}} x_{J_{s-1}} \otimes\, y_{J}x_{J_s}.
$$

\smallskip\noindent
$\bullet$ If $s=4$,
$$
\pm \, a_1(2)^{\alpha_1} b_1(2)^{\beta_1} y_1^{\epsilon_1}y_{J^c}x_{J_1}\otimes a_1(2)^{\alpha_2}y_1^{\epsilon_2}x_{J_2} \otimes b_1(2)^{\beta_3}y_1^{\epsilon_3}x_{J_3} \otimes y_Jx_{J_4}.
$$

\smallskip\noindent
$\bullet$ If $s=3$,
$$
\pm \, a_1(2)^{\alpha_1} b_1(2)^{\beta_1} y_1^{\epsilon_1}y_{J^c}x_{J_1}\otimes a_1(2)^{\alpha_2}y_1^{\epsilon_2}x_{J_2} \otimes b_1(2)^{\beta_3} y_Jx_{J_3}.
$$
In either case, such a term $T$ vanishes in $B_g$ unless each of the following conditions holds:
\begin{enumerate}
\item $J=\{1\}$ or $J=\{1,\ldots,n\}$. 

(Indeed, if $1\not\in J$, then the non-triviality of $T$ in $B_g$ forces $\alpha_1=\beta_1=\epsilon_1=0$, so that $\alpha_2=\beta_3=1$ and $\epsilon_i=1$ for $2\leq i\leq s-1$, which is impossible since $a_1(2) y_1=0$. Thus $1\in J$ must hold. Furthermore, $2$ lies in $s-1$ of the sets $J_1,\ldots,J_{s}$ so, in particular, $x_2$ shows up either in the first tensor factor of $T$ (where $y_{J^c}$ appears), or in the last tensor factor of $T$ (where $y_J$ appears). Therefore, the reduced form of the defining relations in $B_g$ and the non-triviality of $T$ in $B_g$ force either $J-\{1\}=\varnothing$, or $J-\{1\}=\{2,\ldots,n\}$.)

\item $1\not\in J_1$, so that $1\in J_i$ for $2\leq i\leq s$. 

(Indeed, if $1\in J_1$, the non-triviality of $T$ in $B_g$ forces $\alpha_1=0=\beta_1$, so $\alpha_2=1=\beta_3$. But this is incompatible with the non-triviality of $T$  in $B_g$ and the fact that $1$ must lie in either $J_2$ or $J_3$.)

\item $\alpha_2=0=\beta_3$, so that $\alpha_1=1=\beta_1$. 

(For we have just noted that $1\in J_2\cap J_3$.)

\item $\epsilon_1=0$, so that $\epsilon_i=1$ for $2\leq i\leq s-1$. 

(For we have just noted that $\alpha_1=1=\beta_1$.)
\end{enumerate}
Further, when $J=\{1\}$, the term $T$ vanishes in $B_g$ unless $J_1=\varnothing$  (the inclusion $J_1\subseteq\{1\}$ follows by looking at the first tensor factor of $T$ and the relations defining $B_g$, whereas the actual equality $J_1=\varnothing$ follows from condition 2 above) and, therefore, $J_i=\{1,\ldots,n\}$ for $2\leq i\leq s$. Thus, the only such $T$ with (potentially) non-vanishing image in $B_g$ is, up to a sign,
\begin{align}
a_1(2)b_1(2)y_2\cdots y_n{}\otimes{} & y_1 x_1\cdots x_n\otimes\cdots\otimes y_1 x_1\cdots x_n\nonumber\\&=\pm\omega_1 y_2\cdots y_n\otimes\omega_1 x_2\cdots x_n \otimes \cdots\otimes \omega_1 x_2\cdots x_n.\label{primersumando}
\end{align}
Likewise, when $J=\{1,\ldots,n\}$, the term $T$ vanishes in $B_g$ unless $J_s=\{1\}$ (the inclusion $J_s\subseteq\{1\}$ follows by looking at the last tensor factor of $T$ and the relations defining $B_g$, whereas the actual equality $J_s=\{1\}$ follows from condition 2 above) and $J_1=\{2,\ldots,n\}$ while $J_i=\{1,\ldots,n\}$ for $2\leq i\leq s-1$ (in view of condition 2 above and the properties of the $J_i$'s).  Thus, the only such $T$ with (potentially) non-vanishing image in $B_g$ is, up to a sign,
\begin{align}
a_1(2)b_1(2)x_2\cdots x_n{}\otimes{} & y_1 x_1\cdots x_n\otimes\cdots\otimes y_1 x_1\cdots x_n\otimes y_1\cdots y_nx_1\nonumber\\&=\pm\omega_1 x_2\cdots x_n\otimes\cdots\otimes \omega_1 x_2\cdots x_n \otimes \omega_1 y_2\cdots y_n.\label{segundosumando}
\end{align}
Consequently, the image in $B_g$ of the product under consideration is the sum of the term in~(\ref{primersumando}) and the term in~(\ref{segundosumando}), which is non-zero by Theorem~\ref{indlin}.
\end{proof}

\section{Proof of Theorem~\ref{indlin}}
In view of the particularly simple presentation~(\ref{ecopre}) of $B_g$, it might be tempting to guess the form of an additive basis for $B_g$ which, in addition, could easily imply Theorem~\ref{indlin}. However, a few unexpected relations holding in $B_g$ are hidden in $\mathcal{J}_g$. It is the purpose of this section to uncover, in the most efficient way (for the purpose of proving Theorem~\ref{indlin}), some of these unexpected relations.

\medskip
Recall the additive basis $\beta'_2$ of $A_g$ in Corollary~\ref{baseprimamodificada}, that is, the set of products $v_1\cdots v_n$ satisfying the two conditions:
\begin{enumerate}[(i)]
\item For each $i\in\{1,\ldots,n\}$, the factor $v_i$ belongs to $\{1,x_i(p), y_i(p), \omega_i\colon 1\leq p\leq g\}$.
\item At most one of $v_1,\ldots, v_n$ belongs to $\{x_i(p),y_i(p),\omega_i\colon 1\leq i\leq n\mbox{ and }\,2\leq p\leq g\}$.
\end{enumerate}

The verification of the following two lemmas is a straightforward and, thus, omitted task.

\begin{lema}\label{calculosdirectos1}
Let $2\leq j\leq n$. For $v_1\cdots v_n\in\beta'_2$, the product $v_1\cdots v_n\cdot x_jy_j$ vanishes in $A_g$ provided either one of the following conditions holds:
\begin{enumerate}[(i)]
\item $v_j\in\{x_j(p),y_j(p),\omega_j\,\colon1\leq p\leq g\}$.
\item $v_1\in\{x_1(p),y_1(p),\omega_1\,\colon2\leq p\leq g\}$.
\item $v_1\in\{x_1,y_1\}\,\mbox{ and }\,v_k\in\{x_k(p),y_k(p),\omega_k\,\colon2\leq p\leq g\}$ for some $k\not\in\{1,j\}$.
\end{enumerate}
Furthermore, the following relations hold in $A_g$:
\begin{enumerate}[(i)]\addtocounter{enumi}{3}
\item $x_1\cdot x_jy_j=x_1\omega_j+\omega_1x_j$.
\item $y_1\cdot x_jy_j=y_1\omega_j+\omega_1y_j$.
\item $z_k\cdot x_jy_j=z_ky_1x_j-z_kx_1y_j$, \ for $z_k\in\{x_k(p),y_k(p),\omega_k\colon2\leq p\leq g\}$ with $k\not\in\{1,j\}$.
\end{enumerate}
\end{lema}

\begin{lema}\label{calculosdirectos2}
Let $i,j\in\{2,\ldots,n\}$ with $i\neq j$. Then, in $A_g$:
\begin{enumerate}[(1)]
\item The only non-trivial products $z_i\cdot x_iy_j$ with $z_i\in\{x_i(p),y_i(p),\omega_i\colon 1\leq p\leq g\}$ are
\begin{enumerate}[(i)]
\item $y_i\cdot x_iy_j=-\omega_i y_j+\omega_1y_j-y_1x_iy_j+x_1y_iy_j$.
\item $z_i\cdot x_iy_j=-z_ix_1y_j$, \ for $z_i\in\{x_i(p),y_i(p),\omega_i\colon 2\leq p\leq g\}$.
\end{enumerate}
\item The only non-trivial products $z_j\cdot x_iy_j$ with $z_j\in\{x_j(p),y_j(p),\omega_j\colon 1\leq p\leq g\}$ are
\begin{enumerate}[(i)]\addtocounter{enumii}{2}
\item $x_j\cdot x_iy_j=-x_i\omega_j+x_i\omega_1+y_1x_ix_j-x_1x_iy_j$.
\item $z_j\cdot x_iy_j=-z_jx_iy_1$, \ for $z_j\in\{x_j(p),y_j(p),\omega_j\colon 2\leq p\leq g\}$.
\end{enumerate}
\item The only non-trivial product $z_iz_j\cdot x_iy_j$ with $z_i$ and $z_j$ as in~(1) and~(2) above is
\begin{enumerate}[(i)]\addtocounter{enumii}{4}
\item $y_ix_j\cdot x_iy_j=y_1\omega_ix_j+y_1x_i\omega_j-x_1\omega_iy_j-x_1y_i\omega_j+\omega_1y_ix_j-\omega_1x_iy_j$.
\end{enumerate}
\item The only non-trivial products $z_1z_i\cdot x_iy_j$ with $z_1\in\{x_1(p),y_1(p),\omega_1\colon 1\leq g\leq p\}$ and $z_i$ as in~(1) above are
\begin{enumerate}[(i)]\addtocounter{enumii}{5}
\item $x_1y_i\cdot x_iy_j=-x_1\omega_iy_j-\omega_1x_iy_j$.
\item $y_1y_i\cdot x_iy_j=-y_1\omega_iy_j-\omega_1y_iy_j$.
\end{enumerate}
\item The only non-trivial products $z_1z_j\cdot x_iy_j$ with $z_j$ and $z_1$ as in~(2) and~(4) above are
\begin{enumerate}[(i)]\addtocounter{enumii}{7}
\item $x_1x_j\cdot x_iy_j=-x_1x_i\omega_j+\omega_1x_ix_j$.
\item $y_1x_j\cdot x_iy_j=-y_1x_i\omega_j+\omega_1x_iy_j$.
\end{enumerate}
\item All products $z_1z_iz_j\cdot x_iy_j$ with $z_i$, $z_j$ and $z_1$ as in~(1),~(2) and~(4) vanish.
\end{enumerate}
\end{lema}

Set $\gamma_2=\beta'_2-\gamma_1$, where $\gamma_1\subseteq\beta_2'$ consists of the products $v_1\cdots v_n$ satisfying either one of the following two conditions:
\begin{enumerate}[(i)]\addtocounter{enumi}{2}
\item There is a unique $i\in\{1,\cdots,n\}$ for which $v_i=\omega_i$ and $v_j=x_j$ for $j\neq i$.
\item There is a unique $i\in\{1,\cdots,n\}$ for which $v_i=\omega_i$ and $v_j=y_j$ for $j\neq i$.
\end{enumerate}
There is an obvious additive splitting $A_g=C_{g,1}\oplus C_{g,2}$, where $C_{g,\epsilon}$ is the additive span of $\gamma_\epsilon$ ($\epsilon=1,2$). The final technical task in this paper, the proof of Theorem~\ref{indlin}, will be accomplished below by arguing first that the ideal $\mathcal{J}_g$ defining $B_g$ preserves the above splitting, i.e.~by giving an additive decomposition
\begin{equation}\label{artdeco}
\mathcal{J}_g=\mathcal{J}_{g,1}\oplus \mathcal{J}_{g,2},
\end{equation}
where $\mathcal{J}_{g,\epsilon}$ is a vector subspace of $C_{g,\epsilon}$ ($\epsilon=1,2$), and then by giving a description of the (additive structure of the) quotient $C_{g,1}/\mathcal{J}_{g,1}$, for which a basis will clearly be given by the two elements in the statement of Theorem~\ref{indlin}.

\medskip
In what follows, an element $\mathbf{v}=v_1\cdots v_n\in\beta'_2$, will be denoted as
\begin{itemize}
\item $\mathbf{v}(0)$ to indicate that $v_k\in\{1,x_k,y_k\}$ for all $k=1,\ldots,n$; 
\item $\mathbf{v}(i_1,\ldots,i_t)$, for $i_1,\ldots,i_t\in\{1,\ldots,n\}$, to indicate that $v_{i_k}=1$ for $k\in{1,\ldots,t}$. 
\end{itemize}
These two conventions will also be combined. For instance, by writing $\mathbf{v}(0,1,j)$ we mean that the element $\mathbf{v}\in\beta'_2$ satisfies $v_k\in\{1,x_k,y_k\}$ for all $k=1,\ldots,n$, as well as $v_1=v_j=1$.

\begin{proof}[Proof of Theorem~\ref{indlin}]
A set of additive generators of $\mathcal{J}_g$ is given by the products  $\mathbf{v}\cdot r$ with $\mathbf{v}=v_1\cdots v_n\in\beta'_2$ and $r\in\{x_iy_j\,\colon i,j\in\{2,\ldots,n\}\}$. The additive decomposition~(\ref{artdeco}) will follow once we check that
\begin{equation}\label{laasevn}\mbox{\begin{minipage}{13.1cm}
{\em the expression of each such product $\mathbf{v}\cdot r=v_1\cdots v_n \cdot x_iy_j$ (in terms of the basis $\beta'_2$) involves either only elements of $\gamma_1$ or, else, only elements of $\gamma_2$.}
\end{minipage}}\end{equation}
{\bf Case $i=j\geq2$}. By Lemma~\ref{calculosdirectos1}(i), we only need to consider products $\mathbf{v}(j)\cdot x_jy_j$. Recalling from~(\ref{unanueva}) that $x_jy_j=\omega_j-\omega_1+y_1x_j-x_1y_j$, it is clear that (\ref{laasevn}) holds, with $\gamma_2$ being the relevant basis, if $\mathbf{v}=\mathbf{v}(1,j)$ ---in checking this type of assertions, the reader might find it convenient to consider first the case $\mathbf{v}=\mathbf{v}(0,1,j)$. Thus, by Lemma~\ref{calculosdirectos1}(ii) and~(iii), we can assume $v_1\in\{x_1,y_1\}$ and $\mathbf{v}=\mathbf{v}(0)$. In other words, it remains to consider products of the form 
$$
x_1\mathbf{v}(0,1,j)\cdot x_jy_j\quad\,\mbox{and }\quad y_1\mathbf{v}(0,1,j)\cdot x_jy_j.
$$
It is clear from Lemma~\ref{calculosdirectos1}(iv) and~(v) that~(\ref{laasevn}) holds true for the two types of products just described, and that the only such products whose expression in terms of the basis $\beta'_2$ involves only elements from $\gamma_1$ can actually be written, up to a sign, as
\begin{equation}\label{fin1}
\omega_1x_2\cdots x_n+(-1)^jx_1x_2\cdots x_{j-1}\omega_j x_{j+1}\cdots x_n
\end{equation}
and
\begin{equation}\label{fin2}
\omega_1y_2\cdots y_n+(-1)^jy_1y_2\cdots y_{j-1}\omega_j y_{j+1}\cdots y_n.
\end{equation}

\smallskip\noindent
{\bf Case $i,j\in\{2,\ldots,n\}$ with $i\neq j$}. It is obvious that~(\ref{laasevn}) holds, with $\gamma_2$ being the relevant basis, provided $\mathbf{v}=\mathbf{v}(i,j)$. The rest of the possibilities can be analyzed on a term-by-term basis, depending on the values of $z_i$ and $z_j$ in a product $z_iz_j\mathbf{v}(i,j)\cdot x_iy_j$, where $z_t\in\{1,x_t(p),y_t(p),\omega_t\colon 1\leq p\leq g\}$. Actually, by Lemma~\ref{calculosdirectos2}, the only factors involved in the expression of any $z_iz_j\cdot x_iy_j$ can come from the coordinates 1, $i$ and $j$. Therefore it is convenient to split the analysis by considering the products
\begin{equation}\label{2fnls}
z_iz_j\mathbf{v}(1,i,j)\cdot x_iy_j\quad\mbox{and}\quad z_1z_iz_j\mathbf{v}(1,i,j)\cdot x_iy_j.
\end{equation}
Lemma~\ref{calculosdirectos2} describes the expression of the corresponding factors $z_iz_j\cdot x_iy_j$ and $z_1z_iz_j\cdot x_iy_j$ in terms of the basis $\beta'_2$. In all such cases one checks, by direct inspection, that
\begin{itemize}
\item (\ref{laasevn}) holds true for all products in~(\ref{2fnls}),
\item the only products in~(\ref{2fnls}) whose expression in terms of $\beta'_2$ involves elements from $\gamma_1$ are those arising from instances~(vii) and~(viii) of Lemma~\ref{calculosdirectos2}, in which case
\item the resulting expressions in terms of the basis $\beta'_2$ coincide with those in~(\ref{fin1}) and ~(\ref{fin2}) ---note that signs in items~(vii) and~(viii) of Lemma~\ref{calculosdirectos2} are important here!
\end{itemize} 

The proof is complete since the above considerations imply that the decomposition~(\ref{artdeco}) holds in such a way that an additive basis for the resulting additive summand $C_{g,1}/\mathcal{J}_{g,1}$ of $B_g$ is given by the two elements in the statement of Theorem~\ref{indlin}.
\end{proof}

\section{The case $s=2$}
As mentioned in the introduction, the case $s=2$ in Theorem~\ref{main} reduces to Theorem~A in~\cite{CoFa11}. We have given full proof details for that case too because we believe that there are a couple of weak points and, most critically, at least one flawed argument in the homological part of Cohen-Farber's argument. This section describes such potential problems. The reader is assumed to be familiar with the notation in~\cite{CoFa11}. 

\medskip
The main problem happens at the end of the fourth paragraph of the proof of~\cite[Theorem~5.1]{CoFa11}, where the authors assert that the proof of the case for genus $g\geq2$ can be reduced to the consideration of the $g=2$ case by ``annihilating all generators of the form $1\times\cdots\times u\times\cdots\times1$ where $u\in\{a(q),b(q) \colon 3\leq q\leq g\}$''. (Note the typo ``$3\leq q\leq n$'' in~\cite{CoFa11}.) Such an argument does not work because if, for instance, we set $a(3)=0$ in the $i$-th axis, then $w=a(3)b(3)$ would also be zero in that axis. But this interferes  (for $i=1$) with Cohen and Farber's later calculation using the non-triviality of $\omega_1$ (see the last displayed formula in the proof of~\cite[Theorem~5.1]{CoFa11}).

\medskip
In addition, we believe that a weak argument arises at the end of the proof of~\cite[Theorem~5.1]{CoFa11}, where the authors assert that
\begin{equation}\label{notablyweak}\mbox{
\begin{minipage}{5.5in}
the non-zero term $\pm2\omega_1y_2y_3\cdots y_n\otimes\omega_1x_2x_3\cdots x_n$ arises in the expansion of the product $\bar{a}_1\bar{b}_1\bar{c}_1\bar{d}_1\prod_{j=2}^n\bar{x}_j\bar{y}_j$ in such a way that \emph{no other summand in the expansion involves this (non-zero) tensor product.}
\end{minipage}
}\end{equation}
The (apparently implicit) argument supporting~(\ref{notablyweak}) is based on two facts noted in earlier parts of Cohen-Farber's paper: 
\begin{itemize}
\item [(I)] On the one hand, as indicated at the end of the proof of~\cite[Theorem~4.1]{CoFa11} (i.e.~when dealing with the algebra $A_T$ in the genus-$1$ case), the expansion (in terms of basis elements) of $\prod_{j=1}^n\bar{x}_j\bar{y}_j$ uses (with coefficient $\pm1$) the basis element $y_1y_2y_3\cdots y_n\otimes x_1x_2x_3\cdots x_n$.
\item[(II)] On the other hand, near the bottom of page~656 of~\cite{CoFa11}, it is observed (without further explanation, though) that ``The subalgebra of $B_{\Sigma}$ generated by $\{a_i,b_i\colon 1\leq i\leq n\}$ is isomorphic to the subalgebra $A_T$ arising in the genus one case''.
\end{itemize}

The problem is that the latter two facts do not really support~(\ref{notablyweak}) for, although $A_T$ were a honest subalgebra of $B_{\Sigma}$, nothing is said about the (potential) injectivity of the obvious map $(2\omega_1\otimes\omega_1)\cdot A_T\to B_{\Sigma}$. In the Cohen-Farber approach, fixing these problems requires, in principle, an explicit description of additive bases for the subquotient algebras they deal with. Such a task tends to become combinatorially involved, specially in the case of Rudyak's higher TC. We have greatly simplify the job by working in a much smaller subquotient---small enough to detect just the minimal needed information.

\medskip
Also worth remarking is what appears to the authors of this paper to be a weak statement of item~(ii) in~\cite[Lemma~2.1]{CoFa11}, namely, the assertion that an epimorphic image $B$ of an algebra $A$ over a field has $\zcl(A)\geq\zcl(B)$. The verification of such a property is left as a ``straightforward exercise'' in~\cite{CoFa11} and, as in the case of the dual statement in item~(i), its proof should naturally start by picking zero-divisors $b_1,\ldots,b_t\in B\otimes B$ with $b_1\cdots b_t\neq0$. With these conditions it is certainly obvious that, for any choice of preimages $a_i\in A\otimes A$ of each $b_i$, the product $a_1\cdots a_n$ is forced to be non-zero. But the point is to make sure that each $a_i$ can be chosen to be a zero-divisor in $A$, which does not seem to be accomplishable in the stated generality. Nonethless, what can certainly be done (and has been done in this paper) is to argue the non-triviality of some given product of zero-divisors in $A\otimes A$ by exhibiting the non-triviality of the image of the product in $B\otimes B$.


\begin{thebibliography}{10}

\bibitem{bgrt}
Ibai Basabe, Jes{\'u}s Gonz{\'a}lez, Yuli~B. Rudyak, and Dai Tamaki.
\newblock Higher topological complexity and its symmetrization.
\newblock {\em Algebr. Geom. Topol.}, 14(4):2103--2124, 2014.

\bibitem{CoFa11}
Daniel~C. Cohen and Michael Farber.
\newblock Topological complexity of collision-free motion planning on surfaces.
\newblock {\em Compos. Math.}, 147(2):649--660, 2011.

\bibitem{MR513543}
F.~R. Cohen and L.~R. Taylor.
\newblock Computations of {G}el\cprime fand-{F}uks cohomology, the cohomology
  of function spaces, and the cohomology of configuration spaces.
\newblock In {\em Geometric applications of homotopy theory ({P}roc. {C}onf.,
  {E}vanston, {I}ll., 1977), {I}}, volume 657 of {\em Lecture Notes in Math.},
  pages 106--143. Springer, Berlin, 1978.

\bibitem{trmpfarber}
Michael Farber.
\newblock Topology of robot motion planning.
\newblock In {\em Morse theoretic methods in nonlinear analysis and in
  symplectic topology}, volume 217 of {\em NATO Sci. Ser. II Math. Phys.
  Chem.}, pages 185--230. Springer, Dordrecht, 2006.

\bibitem{MR2455573}
Michael Farber.
\newblock {\em Invitation to topological robotics}.
\newblock Zurich Lectures in Advanced Mathematics. European Mathematical
  Society (EMS), Z\"urich, 2008.

\bibitem{JeMa14}
Jes{\'u}s Gonz{\'a}lez and Mark Grant.
\newblock Sequential motion planning of non-colliding particles in {E}uclidean
  spaces.
\newblock {\em Proc. Amer. Math. Soc.}, 143(10):4503--4512, 2015.

\bibitem{gggl}
Jes\'us Gonz\'alez, B\'arbara Guti\'errez, Darwin Guti\'errez, and Adriana
  Lara.
\newblock Motion planning in real flag manifolds.
\newblock {\em Accepted for publication in { \em Homology, Homotopy and
  Applications}}.

\bibitem{LuptonSche13}
Gregory Lupton and J{\'e}r{\^o}me Scherer.
\newblock Topological complexity of {$H$}-spaces.
\newblock {\em Proc. Amer. Math. Soc.}, 141(5):1827--1838, 2013.

\bibitem{Ru10}
Yuli~B. Rudyak.
\newblock On higher analogs of topological complexity.
\newblock {\em Topology Appl.}, 157(5):916--920, 2010.

\bibitem{MR1404924}
Burt Totaro.
\newblock Configuration spaces of algebraic varieties.
\newblock {\em Topology}, 35(4):1057--1067, 1996.

\end{thebibliography}

\def\cprime{$'$}

\bigskip\sc
Departamento de Matem\'aticas

Centro de Investigaci\'on y de Estudios Avanzados del IPN

Av.~IPN 2508, Zacatenco, M\'exico City 07000, M\'exico

{\tt jesus@math.cinvestav.mx}

{\tt
bgutierrez@math.cinvestav.mx}

\end{document}